\documentclass[11pt]{amsart}

\usepackage{amssymb,amsmath,enumerate,epsfig}
\usepackage{amsfonts,color}
\usepackage{a4,latexsym}
\usepackage{parskip}
\usepackage[all]{xy}
\usepackage{hyperref}
\hypersetup{pdftitle=dyn}

\newcommand{\C} {\mathbb{C}}
\newcommand{\Q} {\mathbb{Q}}
\newcommand{\N}  {\mathbb{N}}

\newcommand{\F}{\mathbb{F}}
\newcommand{\Z}{\mathbb{Z}}
\newcommand{\p}{\mathfrak{p}}

\newcommand{\OO}{\mathcal{O}}

\newcommand{\PP}{\mathbb{P}}
\newcommand{\NS}{\mathop{\rm NS}}

\newcommand{\MW}{\mathop{\rm MW}}
\newcommand{\MWL}{\mathop{\rm MWL}}
\newcommand{\Km}{\mathop{\rm Km}}
\newcommand{\disc}{\mathop{\rm disc}}

\newcommand{\Aut}{\mathop{\rm Aut}}

\newcommand{\Triv}{\mathop{\rm Triv}}

\newcommand{\mX}{\mathcal X}

\newcommand{\rk}{\mathop {\rm rk}}

\newcommand{\Hom}{\mathop{\rm Hom}}

\newtheorem{Theorem}{Theorem}[section]
\newtheorem{Proposition}[Theorem]{Proposition}
\newtheorem{Lemma}[Theorem]{Lemma}

\newtheorem{Corollary}[Theorem]{Corollary}

\theoremstyle{remark}

\newtheorem{Remark}[Theorem]{Remark}

\theoremstyle{definition}

\begin{document}

\title{Dynamics on supersingular K3 surfaces}

\dedicatory{Dedicated to Tetsuji Shioda on the occasion of his 75th birthday}

\author{Matthias Sch\"utt}
\address{Institut f\"ur Algebraische Geometrie, Leibniz Universit\"at
  Hannover, Welfengarten 1, 30167 Hannover, Germany}
\email{schuett@math.uni-hannover.de}

\thanks{Funding   by ERC StG~279723 (SURFARI) 
 is gratefully acknowledged.}

\date{May 27, 2015}

\begin{abstract}
For 
any odd
characteristic $p\equiv 2\mod 3$,
we exhibit an explicit automorphism on the supersingular K3 surface of Artin invariant one
which does not lift to any characteristic zero model.
Our construction builds on elliptic fibrations to produce a closed formula
for the automorphism's characteristic polynomial on second cohomology,
which turns out to be an irreducible Salem polynomial of  degree 22 with  coefficients varying with $p$.
\end{abstract}
%
%
 \maketitle

 \section{Introduction}
 \label{s:intro}

Recently, there has been a burst of activity on dynamics of K3 surfaces,
in particular for supersingular ones in positive characteristic. 
This paper contributes to this area with rather explicit results.

The entropy of automorphisms of algebraic varieties (or K\"ahler manifolds over $\C$)
is related to Salem polynomials.
Often one aims for  Salem numbers
which are either small or have large degree.
On K3 surfaces, the maximum degree equals the second Betti number $b_2=22$,
but it can only be attained in two specific settings:
either on non-projective complex K3 surfaces which contain no algebraic curves at all, as studied by McMullen \cite{McM},
or on supersingular K3 surfaces in characteristic $p>0$.

A specific feature of automorphisms of maximal Salem degree 
on supersingular K3 surfaces
was pointed out by Esnault and Oguiso \cite{EO}:
such an automorphism cannot lift to any characteristic zero model of the K3 surface.
Esnault and Oguiso illustrated this phenomenon explicitly with the Fermat quartic  in characteristic $3$,
building on work of Kond\=o and Shimada \cite{KS}.
On the implicit side, there are non-liftability results
for the supersingular K3 surface $X(p)$ of Artin invariant $\sigma=1$ in characteristic $p\neq 5,7,13$ in \cite{BC}, \cite{EOY}.
Our aim is to exhibit non-liftable automorphisms on
an infinite series of supersingular K3 surfaces in an explicit and systematic manner:

\begin{Theorem}
\label{thm}
For any odd prime $p\equiv 2\mod 3$,
there exists an explicit $g\in\Aut(X(p))$ of Salem degree $22$.
No power $g^r \; (r\in\Z\setminus\{0\})$ lifts to any characteristic zero model of $X(p)$.
\end{Theorem}

The key ingredient of our approach consists in the theory of elliptic fibrations on K3 surfaces
and in particular Mordell-Weil lattices,
since these are at the same time accessible, versatile and allow for explicit descriptions
of automorphisms.
We will review the basics on elliptic fibrations and dynamics
in the next two sections,
before explaining the construction of the automorphisms proving Theorem \ref{thm}.
We point out that the techniques are flexible enough
to lend themselves to the study of dynamics on other K3 surfaces,
both over $\C$ and in positive characteristic,
either of which we hope to pursue in future work.

\section{Elliptic fibrations}
\label{s:ell}

Let $X$ be a K3 surface over an algebraically closed field $k=\bar k$:
\[
\omega_X\cong \OO_X,\;\;\; h^1(X,\OO_X)=0.
\]
Classical examples are smooth quartics in $\PP^3$ or more specifically Kummer surfaces,
i.e. the resolution of the quotient of an abelian surface $A$ by inversion $\imath$ with respect to the group law 
(outside characteristic $2$):
\[
\Km(A) = \widetilde{A/\imath}.
\]
In what follows we will restrict for simplicity to elliptic K3 surfaces,
although much of what is stated holds true in much greater generality (see \cite{SSh}).
An elliptic fibration on $X$ is a surjective morphism
\[
\pi: \; X\to \PP^1
\]
such that the generic fiber is a smooth curve of genus $1$.
Here we will only deal, without further distinction, with jacobian elliptic fibrations,
i.e. the morphism $\pi$ admits a section.
This makes the generic fiber $E$ into an elliptic curve over the function field $k(t)$
where $t$ denotes an affine parameter of $\PP^1$.
One derives an extensive correspondence between $X$ and $E$
as either determines the other (see \cite{Neron}).
Notably this materializes in a bijection between
sections of $\pi$ and rational points of $E$.
Either set thus forms a group which is usually referred to as \emph{Mordell-Weil group} $\MW(X)$,
or if we need to specify the fibration, $\MW(X,\pi)$.
In order to endow $\MW(X)$ with a lattice structure following Shioda \cite{Sh-MW},
we introduce the trivial lattice
\[
\Triv(X) = \langle \text{zero section, fiber components}\rangle \subset\NS(X).
\]
As a consequence of Kodaira's classification of singular fibers \cite{K}
(or of Tate's algorithm for the non-complex setting \cite{Tate}),
the trivial lattice decomposes as an orthogonal sum
of 
\begin{itemize}
\item 
the hyperbolic plane $U$ generated by zero section $O$ and general fiber $F$, and 
\item
negative-definite root lattices of Dynkin type $A_n, D_m, E_l$
generated by the fiber components after omitting the identity component
(i.e. the component meeting $O$).
\end{itemize}
Note that  $\Triv(X)$ is hyperbolic (i.e. of signature $(1,{\rk}\Triv(X)-1)$) and that 
the indices of the root lattices equal the number of fiber components minus one
which leads to a straight forward rank formula for $\Triv(X)$.
The key theorem for (jacobian) elliptic fibrations states
that any divisor in $\NS(X)$ can be written as a sum of horizontal and vertical divisors,
or more precisely, in terms of sections and fiber components:

\begin{Theorem}
[Shioda \cite{Sh-MW}]
\label{thm:E}
There is an isomorphism of groups:
\[
 \MW(X) \cong \NS(X)/\Triv(X).
 \]
\end{Theorem}

It is precisely this isomorphism
which will allow us to make both many automorphisms on K3 surfaces
and their induced actions on cohomology (or $\NS(X)$) explicit.
To this end, we only have to add a lattice structure on the Mordell-Weil group,
or rather its quotient by the torsion subgroup,
by means of the orthogonal projection in $\NS(X)\otimes\Q$ with respect to $\Triv(X)$.
Reversing the sign of the intersection pairing, this makes
\[
\MWL(X) = \MW(X)/\text{torsion}
\]
into a positive definite lattice (though not necessarily integral).

\section{Automorphisms and Salem numbers}
\label{s:salem}

In this section, we will consider a K3 surface $X$ equipped with an automorphism $g\in\Aut(X)$.
For instance, one can think of a linear transformation of the ambient projective space
which preserves $X$.
However, such an automorphism is necessarily of finite order since it leaves the hyperplane section invariant
and thus acts on its orthogonal complement in $\NS(X)$, a negative-definite lattice,
whence  the claim can be seen as a consequence of  the Torelli theorem
(although it holds true in greater generality).
Given an elliptic fibration, it is equally instructive to consider translation by a rational point $P$ on the generic fiber $E$
which extends to an automorphism of $X$, the order of which equals the order of $P\in E(k(t))$.

In what follows, we work with $\ell$-adic \'etale cohomology
$H^2_\text{\'et}(X,\Q_\ell(1))$ after applying a Tate twist
(in characteristic $p\neq \ell$ -- over $\C$, we could equally well work with singular cohomology). 
For shortness, we will only write $H^2(X)$
and compute the characteristic polynomial 
\[
\mu(g^*; H^2(X)) \in \Z[x]
\]
(which is indeed independent of $\ell$ and integral by general theory).
The relation with dynamics is fostered by the occurrence of Salem polynomials.
A monic irreducible polynomial $f\in\Z[x]$ of degree $2d$ is called \emph{Salem polynomial}
if it has $2d-2$ roots on the unit circle plus two real positive roots $\alpha, 1/\alpha$;
as a convention, we denote the root with absolute value greater than $1$ by $\alpha$.

\begin{Theorem}
\label{thm:Salem}
The characteristic polynomial
$\mu(g^*; H^2(X))$ factors into cyclotomic polynomials and at most one Salem polynomial.
\end{Theorem}

Over $\C$, the result is due to McMullen \cite{McM}.
In positive characteristic,
it relies on essential input of a result by Esnault and Srinivas  \cite{ES}
which implies that $g^*$ is finite on the orthogonal complement of the 
span  of all the $g^*$-iterates of any polarization on $X$.
In particular, $g^*$ can only have positive entropy on a subspace of $\NS(X)$
which then necessarily leads to a Salem polynomial by \cite[Prop. 3.1]{EO}
(which again draws  on \cite{McM}).

We can now define the  entropy of $g$:
\[
h(g) = \begin{cases} 0 & \text{ if $\mu(g^*; H^2(X))$ factors completely into cyclotomic polynomials,}\\
\log \alpha & \text{ if  there is a Salem polynomial $f\mid\mu(g^*; H^2(X))$ 
with real root $\alpha>1$ }
\end{cases}
\]
We point out that this notion is consistent with the topological entropy from complex dynamics
(compare \cite{ES}).

For our purposes, it will be crucial that $\NS(X)\otimes\Q_\ell$ embeds as a direct summand into $H^2(X)$
via the cycle class map.
Denote the orthogonal complement of $\NS(X)\otimes\Q_\ell$ with respect to cup-product by
\[
T_\ell(X) = (\NS(X)\otimes\Q_\ell)^\perp\subset H^2(X).
\]
This notation is consistent with the transcendental lattice $T(X)\subset H^2(X,\Z)$ of a complex K3 surface $X$
as  in that case $T_\ell(X) = T(X)\otimes \Q_\ell$.
Clearly $g^*$ preserves $\NS(X)$,
so the above direct sum decomposition of $H^2(X)$  is compatible with the $g^*$-action,
and we obtain a factorisation
\[
\mu(g^*; H^2(X)) = \mu(g^*; T_\ell(X)) \, \mu(g^*; \NS(X))
\]
over $\Z$ (since $\NS(X)$ is a lattice over $\Z$).
The crucial dynamical restriction for our purposes is the following observation
which seems to be due to Oguiso in the complex case;
for positive characteristic,
the argument has been sketched just below Theorem \ref{thm:Salem}.

\begin{Theorem}
\label{thm:NS}
If $X$ is algebraic, then a Salem polynomial can occur only on $\NS(X)$,
i.e. it necessarily divides $\mu(g^*; \NS(X))$.
\end{Theorem}

Over $\C$, we can be a little more explicit about the other factors:

\begin{Lemma}
\label{lem:T}
If $X$ is a complex algebraic K3 surface,
then $\mu(g^*; T(X))$ is a perfect power of a cyclotomic polynomial.
\end{Lemma}

\begin{proof}
From Theorems \ref{thm:Salem} and \ref{thm:NS},
we know that $\mu(g^*; T(X))$
is a product of cyclotomic polynomials.
But then the transcendental lattice $T(X)$ is endowed with a Hodge structure
which is irreducible over $\Q$ (by definition, since $p_g=1$).
Hence all the irreducible factors are the same.
\end{proof}

\begin{Remark}
\label{rem}
There is an alternative proof without reference to Theorems \ref{thm:Salem} and \ref{thm:NS}:
For an algebraic K3 surface, it is known that any automorphism $g$ 
acts by a root of unity on the regular 2-forms.
By standard comparison theorems,
this root occurs as an eigenvalue of $g^*$ on $H^2(X,\Z)$,
and by the very definition through the Hodge structure, on $T(X)$.
Then the irreducibility gives the claim.
\end{Remark}

Note the immediate consequence of Theorem \ref{thm:NS}
that on an algebraic K3 surface in characteristic zero,
an automorphism can have a Salem factor of its characteristic polynomials
of degree at most $20$ only.
(This is the reason why for Salem degree $22$, previous work was concentrating
on non-projective K3 surfaces, see McMullen \cite{McM}.)

K3 surfaces $X$ in positive characteristic $p>0$ come with the advantage
that they allow for Salem polynomials of degree $22$ attained on $\NS(X)$
if the latter has rank $22$,
i.e. if $X$ is supersingular (in Shioda's sense if $p=2$).
To see this at work,
it seems to suffice with the supersingular K3 surface of Artin invariant $\sigma=1$,
unique up to isomorphism by work of Ogus \cite{Ogus}
and  denoted by $X(p)$.

\begin{Theorem}[Blanc-Cantat, Esnault-Oguiso-Yu, Shimada]
\label{thm:22}
For any prime $p$,
there is an automorphism $g\in\Aut(X(p))$ of Salem degree $22$.
\end{Theorem}

\begin{Corollary}
\label{cor:lift}
The automorphism $g\in\Aut(X(p))$ from Theorem \ref{thm:22}
does not lift to any characteristic zero model of $X(p)$.
\end{Corollary}

We point out that  the automorphisms in Theorem \ref{thm:22}
are mostly implicit,
i.e. there are existence results building on elliptic fibrations (in particular of maximal rank)
and some group theory (see \cite{BC}, \cite{EOY}).
In contrast, until the completion of the first version of this paper it was
only for $p=3$ that there was an explicit $g\in\Aut(X(p))$ of Salem degree $22$ known \cite{EO},
building on the calculation of $\Aut(X(3))$ by Kond\= o and Shimada \cite{KS}.
In the meantime, Shimada used lattice-theoretic decriptions
of  involutions of double sextic models
to equip any supersingular K3 surface (of any Artin invariant!) for the first 1000 primes
with an automorphism of Salem degree $22$ \cite{Shimada-dyn}.

Our aim is to exhibit an infinite series of supersingular K3 surfaces, here $X(p)$ for odd primes $p\equiv 2\mod 3$,
with \emph{explicit} automorphisms of Salem degree $22$.
This will be achieved successively in the following sections.

\section{Singular and supersingular K3 surfaces}
\label{s:K3}

A complex K3 surface $X$ is called \emph{singular} (in the sense of exceptional as opposed to non-smooth)
if its Picard number attains the maximum
\[
\rho(X) = h^{1,1}(X) = 20.
\]
Singular K3 surfaces played an instrumental role in the proof of the Torelli theorem
and of the surjectivity of the period map for K3 surfaces, see \cite{SI}, \cite{SM}.
In particular, this led to the notion of a \emph{Shioda-Inose structure}
which relates $X$ to the  product of two isogenous CM-elliptic curves $E\times E'$
through rational degree $2$ maps to the Kummer surface $\Km(E\times E')$:
 \begin{eqnarray}
 \label{eq:SI}
  \xymatrix{E\times E' \ar@{-->}[dr] && X\ar@{-->}[dl]&\\
 & \Km(E\times E')&& T(E\times E')\cong T(X)}
 \end{eqnarray}


We emphasize that when working over non-closed fields,
the above construction can always be carried out over a certain finite extension
of the ground field
(and that a singular K3 surface always admits a model over the ring class field $H(d)$
where $d<0$ denotes the discriminant of $\NS(X)$ by \cite{S-fields}).
Therefore,
classical CM theory for elliptic curves easily gives the following statement for singular K3 surfaces:

\begin{Proposition}
\label{prop:red}
Let $X$ be a singular K3 surface, defined over some number field $K$.
Let $\p$ be a prime  of good reduction above $p\in\N$.
Then $X\otimes\bar\F_\p$ is supersingular if $p$ is inert in $K$.
\end{Proposition}

More precisely, it is known that any supersingular reduction $X\otimes\bar\F_\p$
will have Artin invariant $\sigma=1$,
i.e. $\NS(X\otimes\bar\F_\p)$ has rank $22$ and discriminant $-p^{2\sigma}$ by \cite{Shimada}.
In consequence, Proposition \ref{prop:red} provides a systematic way
to produce projective models for $X(p)$ for all $p$ in some given arithmetic progressions.
Lattice theoretically,
there is a conceptual formulation for this connection due to Shioda \cite{Sh-Murre}:
$X$ admits a certain elliptic fibration (crucial to  \eqref{eq:SI} and nowadays often referred to as Inose's pencil)
such that
\begin{eqnarray}
\label{eq:Hom}
\MWL(X) \cong \Hom(E,E')[2].
\end{eqnarray}
Here $\Hom(E,E')$ is endowed with a lattice structure
by means of the degree,
and the intersection pairing on $\MWL(X)$ is that of $\Hom(E,E')$ scaled by $2$.
By \cite{HSa}, the isometry \eqref{eq:Hom} can even be made Galois-equivariant. 

\section{Isotrivial elliptic fibration}
\label{s:iso}

It is rather exceptional that the whole construction from the previous section can  be made explicit in terms of equations
including generators of $\NS(X(p))$.
Here we exploit one such instance:
the minimal case $d=-3$ where the singular K3 surface $X$ admits an isotrivial elliptic fibration
which has been studied from other perspectives in \cite{S-Michigan}:
\begin{eqnarray}
\label{eq:3IV*}
X:\;\;\; y^2 + t^2 (t-1)^2 y = x^3.
\end{eqnarray}
The corresponding fibration
\[
\pi: X\to \PP^1_t
\]
has three singular fibers of Kodaira type $IV^*$ at $t=0,1,\infty$ (as long as the characteristic is different from $3$)
and 3-torsion sections $(0,0), (0,-t^2(t-1)^2)$,
so that standard formulae indeed confirm rank and discriminant over $\C$.
More precisely,
fiber components and zero sections generate a sublattice
\[
U+E_6^3 \subset \NS(X)
\]
of index $3$ where the 3-divisible classes can be identified with the 3-torsion sections by Theorem \ref{thm:E}.
For later reference, we fix a basis of $\NS(X)$ starting with  the general fiber
and with further divisors indicated by the numbers in the following diagram 
of $(-2)$-curves on $X$
(where 7a and 7b mean that the 7th basis element is the sum of the two fibre components).

\begin{figure}[ht!]

\setlength{\unitlength}{1.2mm}
\begin{picture}(50,50)(0,5)
  \qbezier(25.000,10.000)(33.284,10.000)
          (39.142,15.858)
  \qbezier(39.142,15.858)(45.000,21.716)
          (45.000,30.000)
  \qbezier(45.000,30.000)(45.000,38.284)
          (39.142,44.142)
  \qbezier(39.142,44.142)(33.284,50.000)
          (25.000,50.000)
  \qbezier(25.000,50.000)(16.716,50.000)
          (10.858,44.142)
  \qbezier(10.858,44.142)( 5.000,38.284)
          ( 5.000,30.000)
  \qbezier( 5.000,30.000)( 5.000,21.716)
          (10.858,15.858)
  \qbezier(10.858,15.858)(16.716,10.000)
          (25.000,10.000)
%

\put(25., 50.){\circle*{1.5}}   
\put(26,47){\small 2}

\put(31.84, 48.79){\circle*{1.5}}  

\put(37.86, 45.32){\circle*{1.5}}  
\put(35.5,43){\small 8}

\put( 42.32, 40.){\circle*{1.5}}  
\put(41.2,36.5){\small 9}

\put(44.70, 33.47){\circle*{1.5}}  
\put(41.1,30.5){\small 12}

\put(44.70, 26.53){\circle*{1.5}}  
\put(40.5,24.5){\small 13}

\put( 42.32, 20.){\circle*{1.5}}  
\put(37.5,18.5){\small 20}

\put(37.86, 14.68){\circle*{1.5}}  
\put(33.5,14.5){\small 19}

\put(31.84, 11.21){\circle*{1.5}} 
\put(28,12){\small 18}
 
\put(  25., 10.){\circle*{1.5}}  
\put(21.5,11.5){\small 15}

\put(18.16, 11.21){\circle*{1.5}}  
\put(16,12.8){\small 16}

\put(12.14, 14.68){\circle*{1.5}}  
\put(11,16.5){\small 17}

\put( 7.68, 20.){\circle*{1.5}}  
\put(5.30, 26.53){\circle*{1.5}}  
\put(6.7,25.5){\small 6}

\put(5.30, 33.47){\circle*{1.5}}  
\put(6.7,32.5){\small 5}

\put( 7.68, 40.){\circle*{1.5}}  
\put(7.8,36.7){\small 4}

\put(12.14, 45.32){\circle*{1.5}}  
\put(13,43){\small 3}
\put(18.16, 48.79){\circle*{1.5}}

\put(25,50){\line(0,-1){40}}
\put(25,40){\circle*{1.5}}  
\put(25,20){\circle*{1.5}} 
\put(26.5,20){\small 14} 


\qbezier( 42.32, 40.)(25,30),( 7.68, 20.)
\qbezier( 42.32, 20.)(25,30),( 7.68, 40.)

\put(33.66,35){\circle*{1.5}}  
\put(34.8,32.5){\small 10}

\put(16.34,25){\circle*{1.5}}  
\put(17.5,22.5){\small 11}

\put(33.66,25){\circle*{1.5}}  
\put(34.8,26){\small 7b}

\put(16.34,35){\circle*{1.5}}  
\put(17.5,36){\small 7a}

\put(22.5,51){$O$}
\put( 43.32, 18.){$Q$}  
\put( 2.18, 18.){$2Q$}
\put( -3, 42.){$(t=0)$}
\put( 43.32, 42.){$(t=\infty)$}
\put(20.6,5.5){$(t=1)$}

\end{picture}
\caption{24 $(-2)$-curves supporting singular fibers and torsion sections of $\pi$}
\label{Fig:1}
\end{figure}
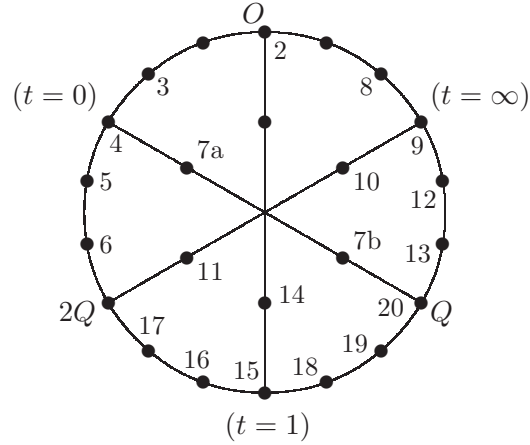

The isotriviality  allows us to pull-back $\pi$ from a rational elliptic surface $S$
by a purely inseparable base change of degree $p$ for any $p\equiv 2\mod 3$
(including $p=2$):
\begin{eqnarray}
\label{eq:RES}
S:\;\;\; y^2 + t (t-1) y = x^3.
\end{eqnarray}
Conceptually, this is a consequence of the specific singular fibers
and their behaviour under base change;
explicitly it can be derived by minimalising the resulting equation
by rescaling $x$ and $y$.
It follows that $X(p)$ is a Zariski surface,
and we obtain $\MWL(X(p) ,\pi)\cong A_2^\vee(p)$
abstractly from the classification by Oguiso-Shioda \cite{OS} through functoriality
\cite[Prop. 8.12]{Sh-MW}.
Alternatively, we can work out generators of $\MW(S)$ explicitly;
for instance, the sections $P_0=(t,t)$ and $\omega P_0=(\omega t,t)$ together with the 3-torsion section $Q=(0,0)$
suffice where $\omega$ simutaneously denotes a primitive third root of unity and the corresponding automorphism
of the generic fiber of  \eqref{eq:3IV*} and \eqref{eq:RES}.
On the elliptic fibration given by \eqref{eq:3IV*}, these induce the following sections, expressed in terms of $p=3n+2$:
\[
P = (t^{n+2}/(t-1)^{2n}, t^2/(t-1)^{3n}),
\;\; \;\; \omega P = (\omega t^{n+2}/(t-1)^{2n}, t^2/(t-1)^{3n}).
\]
One computes the following intersection numbers between sections (symmetric in $P$ and $\omega P$):
\[
P.O = P.Q = P.2Q = n,\;\;\; P.\omega P = 3n.
\]
In addition, $P$ meets the $IV^*$ fibers at $t=\infty$ in the same component as $Q$ ($\# 13$),
at $t=0$ in the other non-identity component compared with $Q$  ($\# 6$)
and at $t=1$ in the identity component.
In terms of the height pairing \cite{Sh-MW} one thus easily verifies that
\[
\langle P,O\rangle = \langle P,Q\rangle = \langle P,2Q\rangle = 0,\;\;\; 
\langle P,P\rangle = 2p/3, \;\;\; \langle P,\omega P\rangle = -p/3,
\]
confirming $\MWL(X(p)) \cong A_2^\vee(p)$.
Translation by $P$ on the generic fiber of \eqref{eq:3IV*} 
gives an automorphism $\boldsymbol{\tau\in\Aut(X(p))}$.
Using Theorem \ref{thm:E}, it is not hard to work out the action of $\tau^*$ on $\NS(X(p))$.
For instance, $\tau^*$ rotates both $IV^*$ fibers at $t=0, \infty$
while leaving the fiber at $t=1$  invariant componentwise.
We illustrate this with two more explicit examples.
Complementing the above basis of $\NS(X)$ by $P, \omega P$ for a $\Z$-basis of $\NS(X(p))$,
we can write
\[
{P+Q = (4,2,-1,-2,-1,0,-2,-2,-4,-3,-2,-3,-2,-3,-6,-4,-2,-5,-4,-2,1,0)}
\]
or
\[
2P = (2 n+6,2,-2,-4,-2,0,-4,-2,-4,-3,-2,-3,-2,-3,-6,-4,-2,-5,-4,-3,2,0) 
\]
Here is a quick guide how to find these representations:
first subtract $O$ to obtain a divisor $D$ with $D.F=1$;
then add and subtract fiber components until $D$ meets each fiber at exactly one component
(partly predicted by the group structure on the smooth locus of the fiber);
finally add a multiple of $F$ such that $D^2=-2$.

In exactly the same way,
we define the automorphism $\boldsymbol{\tau^\omega\in\Aut(X(p))}$
which is given by translation by the section $\omega P$.

The full matrices representing $\tau$ and $\tau^\omega$ in the above basis of $\NS(X(p))$ 
are available from the author's homepage;
the same goes for two further automorphisms
of $X(p)$ which we will develop in the next two sections,
and subsequent calculation data.
Together the four automorphisms will combine for the automorphism 
proving Theorem \ref{thm}.

\section{Alternative elliptic fibration}

One of the keys for the implicit results in \cite{BC}, \cite{EOY}
is the special feature that a  K3 surface
may admit different elliptic fibrations.
For convenience, we shall only work with fibrations
which are already visible on $X$.
By work of Nishiyama \cite{Nishi},
there are 6 such up to isomorphism,
each uniquely encoded in the singular fibers.

In detail, we work with the elliptic fibration which has a fiber of Kodaira type $I_{18}$.
In Figure \ref{Fig:1}, this is visible as the outer circle of $(-2)$-curves.
Note that by general theory, this divisor will in fact induce the fibration
\[
\pi': X\to \PP^1
\]
 in question.
Over $\C$, it is easy to work out $\MW(X,\pi')$:
choosing some curve as zero section $O'$ for $\pi'$,
there is a 3-torsion section $Q'$ and a section $R'$ of height $3/2$
as indicated in Figure \ref{Fig:2}.
By standard formulae for the discriminant \cite[(22)]{SSh}, $Q'$ and $R'$ generate $\MW(X,\pi')$.

\begin{figure}[ht!]

\setlength{\unitlength}{1.2mm}
\begin{picture}(50,50)(0,5)
  \qbezier(25.000,10.000)(33.284,10.000)
          (39.142,15.858)
  \qbezier(39.142,15.858)(45.000,21.716)
          (45.000,30.000)
  \qbezier(45.000,30.000)(45.000,38.284)
          (39.142,44.142)
  \qbezier(39.142,44.142)(33.284,50.000)
          (25.000,50.000)
  \qbezier(25.000,50.000)(16.716,50.000)
          (10.858,44.142)
  \qbezier(10.858,44.142)( 5.000,38.284)
          ( 5.000,30.000)
  \qbezier( 5.000,30.000)( 5.000,21.716)
          (10.858,15.858)
  \qbezier(10.858,15.858)(16.716,10.000)
          (25.000,10.000)
%

\put(25., 50.){\circle*{1.5}}   

\put(31.84, 48.79){\circle*{1.5}}  

\put(37.86, 45.32){\circle*{1.5}}  

\put( 42.32, 40.){\circle*{1.5}}  

\put(44.70, 33.47){\circle*{1.5}}  

\put(44.70, 26.53){\circle*{1.5}}  

\put( 42.32, 20.){\circle*{1.5}}  

\put(37.86, 14.68){\circle*{1.5}}  

\put(31.84, 11.21){\circle*{1.5}} 
 
\put(  25., 10.){\circle*{1.5}}  

\put(18.16, 11.21){\circle*{1.5}}  

\put(12.14, 14.68){\circle*{1.5}}  

\put( 7.68, 20.){\circle*{1.5}}  
\put(5.30, 26.53){\circle*{1.5}}  

\put(5.30, 33.47){\circle*{1.5}}  

\put( 7.68, 40.){\circle*{1.5}}  

\put(12.14, 45.32){\circle*{1.5}}  
\put(18.16, 48.79){\circle*{1.5}}

\put(25,50){\line(0,-1){40}}
\put(25,40){\circle*{1.5}} 
\put(26.5,40){$R'$} 
 
\put(25,20){\circle*{1.5}} 
\put(26.5,20){$2Q'$} 


\qbezier( 42.32, 40.)(25,30),( 7.68, 20.)
\qbezier( 42.32, 20.)(25,30),( 7.68, 40.)

\put(33.66,35){\circle*{1.5}}  
\put(34.8,32.5){$Q'$}

\put(16.34,25){\circle*{1.5}}  
\put(8.5,27){\footnotesize $R'+2Q'$}

\put(33.66,25){\circle*{1.5}}  
\put(32.5,26.5){\footnotesize $R'+Q'$}

\put(16.34,35){\circle*{1.5}}  
\put(17.5,36){$O'$}


\end{picture}
\caption{$I_{18}$ fiber and 6 sections}
\label{Fig:2}
\end{figure}
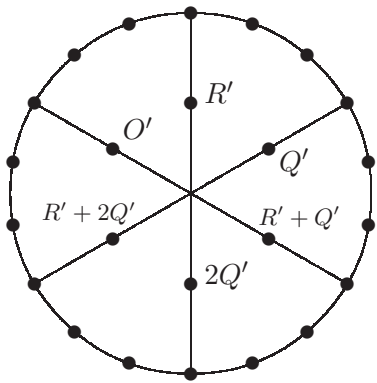

Next we consider $\pi'$ as a fibration on $X(p)$
for an \emph{odd} prime $p=3n+2$ as above.
The congruence assumption on $p$ ensures that the fibration does not degenerate (cf.~Rem.~\ref{rem:p=2}),
and $\MW(X(p),\pi')$ has rank $3$
(cf.~Prop.~\ref{prop:red}).
Generators complementing $Q', R'$ are readily obtained 
from the degree $p$-multisections $P, \omega P$.
By Theorem \ref{thm:E},
these induce sections $P', \omega P'$.
For our purposes, it suffices to describe the sections 
as divisors
in terms of  the multisections and elements from the trivial lattice of $\pi'$
as in section \ref{s:iso}
(subtracting $(p-1)O$ and then proceeding as before).


For completeness we list  the resulting intersection numbers (again symmetric in $P', \omega P'$, and in perfect agreement with $\disc\NS(X(p))=-p^2$):
$$
\begin{array}{lcllcl}
P'.O' & = & (3n^2+8n+1)/4, \;\;\;\;\; &  
P'.Q' & = & (3n+5)  (n+1)/4.\\
P'.R' & = & (3n^2+2n+3)/4, & 
P'.\omega P' & = & 3n.
\end{array}
$$

As in section \ref{s:iso}, we can then define
 an automorphism $\boldsymbol{\tau'\in\Aut(X(p))}$
 by translation by $P'$ on the generic fiber of $\pi'$
and compute the induced action on $\NS(X(p))$.

\begin{Remark}
\label{rem:p=2}
At $p=2$, the fibration $\pi'$ degenerates
as it attains three additional reducible fibers
of Kodaira type $I_2$ (with fiber components given by the sections $-P, -\omega P, -\omega^2 P$).
Since $\MW(X(2),\pi')\cong \Z/6\Z$,
 the above computations cannot carry over.
Indeed, on the contrary, the multisection $P$ induces
the 3-torsion section $Q'$ for $\pi'$.
\end{Remark}


\section{Extra involution}
\label{s:extra}

In order to prove Theorem \ref{thm}, 
it is crucial to throw in an involution
which does not respect the elliptic fibration \eqref{eq:3IV*} above.
Almost to the contrary, it exploits a symmetry in Figure \ref{Fig:1}
which makes the fibration with three fibers of type $IV^*$ visible 
in two essentially different ways: 
the one depicted in Figure \ref{Fig:1} and the one which features the original torsion sections
$O, Q, 2Q$ as triple components of the singular fibers:

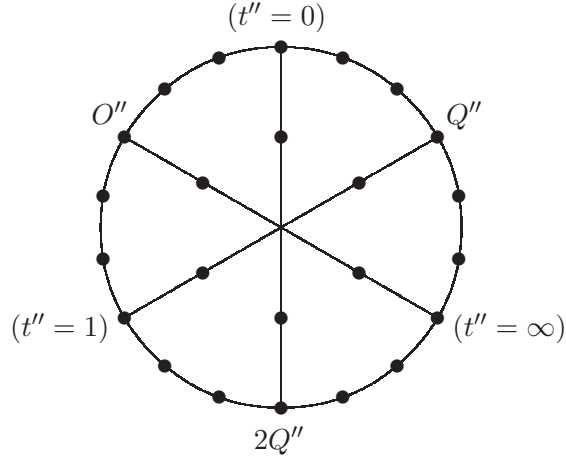
\begin{figure}[ht!]

\setlength{\unitlength}{1.2mm}
\begin{picture}(50,50)(0,5)
  \qbezier(25.000,10.000)(33.284,10.000)
          (39.142,15.858)
  \qbezier(39.142,15.858)(45.000,21.716)
          (45.000,30.000)
  \qbezier(45.000,30.000)(45.000,38.284)
          (39.142,44.142)
  \qbezier(39.142,44.142)(33.284,50.000)
          (25.000,50.000)
  \qbezier(25.000,50.000)(16.716,50.000)
          (10.858,44.142)
  \qbezier(10.858,44.142)( 5.000,38.284)
          ( 5.000,30.000)
  \qbezier( 5.000,30.000)( 5.000,21.716)
          (10.858,15.858)
  \qbezier(10.858,15.858)(16.716,10.000)
          (25.000,10.000)
%

\put(25., 50.){\circle*{1.5}}   

\put(31.84, 48.79){\circle*{1.5}}  

\put(37.86, 45.32){\circle*{1.5}}  

\put( 42.32, 40.){\circle*{1.5}}  

\put(44.70, 33.47){\circle*{1.5}}  

\put(44.70, 26.53){\circle*{1.5}}  

\put( 42.32, 20.){\circle*{1.5}}  

\put(37.86, 14.68){\circle*{1.5}}  

\put(31.84, 11.21){\circle*{1.5}} 
 
\put(  25., 10.){\circle*{1.5}}  

\put(18.16, 11.21){\circle*{1.5}}  

\put(12.14, 14.68){\circle*{1.5}}  

\put( 7.68, 20.){\circle*{1.5}}  
\put(5.30, 26.53){\circle*{1.5}}  

\put(5.30, 33.47){\circle*{1.5}}  

\put( 7.68, 40.){\circle*{1.5}}  

\put(12.14, 45.32){\circle*{1.5}}  
\put(18.16, 48.79){\circle*{1.5}}

\put(25,50){\line(0,-1){40}}
\put(25,40){\circle*{1.5}} 
 
\put(25,20){\circle*{1.5}} 


\qbezier( 42.32, 40.)(25,30),( 7.68, 20.)
\qbezier( 42.32, 20.)(25,30),( 7.68, 40.)

\put(33.66,35){\circle*{1.5}}  

\put(16.34,25){\circle*{1.5}}  

\put(33.66,25){\circle*{1.5}}  

\put(16.34,35){\circle*{1.5}}  

\put(19.0,52.5){$(t''=0)$}
\put( 44.02, 18.){$(t''=\infty)$}  
\put( -5, 18.){$(t''=1)$}
\put( 4, 41.5){$O''$}
\put( 43.32, 41.5){$Q''$}
\put(22,5.5){$2Q''$}

\end{picture}
\caption{Isomorphic fibration with three $IV^*$ fibers}
\label{Fig:3}
\end{figure}

By the uniqueness of the fibration, 
there is an  automorphism $\boldsymbol{\imath\in\Aut(X)}$ (and $X(p)$) switching these two models of
the fibration.
Explicitly, this can be derived after rescaling $x, y$ in \eqref{eq:3IV*}
from the resulting symmetry in $y,t$ (up to sign):
\[
X:\;\;\; (y^2+y)t(t-1) = x^3.
\]
For the new fibration, $P$ and $\omega P$ again induce multisections of degree $p$;
modulo the trivial lattice they are equivalent to $-P, -\omega P$ in terms of the standard basis.
Again, one can thus spell out the induced action of $\imath^*$ on $\NS(X(p))$.

\section{Non-liftable automorphism}

Consider the following automorphism on $X(p)$:
\[
g = \tau'\circ\imath\circ\tau'\circ\imath
\circ\tau^\omega\circ\imath
\circ\tau'\circ\imath\circ\tau'\circ\imath\circ\tau\in\Aut(X(p)).
\]
With a computer algebra system, one computes the characteristic polynomial 
of $g^*$ on $\NS(X(p))$:
\begin{eqnarray}
\label{eq:g}
\mu(g^*) = x^{11}\phi(x+1/x),
\end{eqnarray}
where 
\begin{eqnarray*}
\phi(x) & = &
6804+11016\,n+2187\,{n}^{2}+ \left( -6804\,n-20304+7128\,{n}^{2} \right) x\\
&&
\mbox{}- \left( 16443\,{n}^{2}+64281\,n+34254 \right) {x}^{2}- \left(17442\,{n}^{2}-20925\,n-56020 \right) {x}^{3}\\
&&
\mbox{}+ \left( 107487\,n+31536\,{n}^{2}+63852 \right) {x}^{4}+ \left( 12987\,{n}^{2}-16794\,n-43992 \right) {x}^{5}\\
&&
\mbox{}- \left( 22545\,{n}^{2}+70167\,n+44646 \right) {x}^{6}- \left( 3780\,{n}^{2}-5157\,n-13414 \right) {x}^{7}\\
&&
\mbox{}+ \left( 6696\,{n}^{2}+19467\,n+12900 \right) {x}^{8}+ \left( 378\,{n}^{2}-540\,n-1404 \right) {x}^{9}\\
&&
\mbox{}+ \left( -1308-702\,{n}^{2}-1929\,n \right) {x}^{10} +  {x}^{11}
\end{eqnarray*}

\begin{Lemma}
\label{lem:22}
The polynomial $\psi(x)$ is an  irreducible Salem polynomial over $\Z$
for any $n\in\N$ such that $3n+2$ is prime.
\end{Lemma}

\begin{proof}
Let  $n\in\N$ such that $p=3n+2$ is prime.
Then $\psi(x)$ comes from the automorphism $g\in\Aut(X(p))$
sketched above.
If $\psi(x)$ were not irreducible over $\Z$,
then it would split off a cyclotomic factor by Theorem \ref{thm:Salem}.
This, however, can  be falsified simultaneously for all $n\in\N$
by computing the remainders after division with each cyclotomic polynomial
of degree at most $22$.
\end{proof}

\subsection*{Proof of Theorem \ref{thm}}

We have verified
for any odd prime $p\equiv 2\mod 3$  in Lemma \ref{lem:22}
that $g\in\Aut(X(p))$ is an (irreducible) Salem polynomial of degree $22$. 
For the sake of completeness, we recall 
why $g$ does not lift to any characteristic zero model of $X(p)$.
Assume to the contrary that there is a lift $\mX$ of $X(p)$ 
over some field $K$ of characteristic zero,
with some automorphism $\tilde g\in\Aut(\mX)$ lifting $g$.
Since the isomorphism $H^2(\mX_{\bar K})\cong H^2(X(p)_{\bar \F_p})$ is equivariant for the action of $\tilde g^*$ resp. $g^*$,
this implies
\[
\mu(\tilde g^*) = \mu(g^*).
\]
By Theorem \ref{thm:NS}, the Salem factor $\phi(x)$ can only be attained
on $\NS(\mX_{\bar K})$,
hence 
\[
\rho(\mX_{\bar K})=\deg\phi(x) = 22.
\]
This contradicts Lefschetz' bound $\rho(\mX_{\bar K})\leq h^{1,1}(\mX_{\bar K})=20$.
The same reasoning applies to any non-trivial power of $g$,
since the characteristic polynomial remains irreducible of degree $22$, and in particular Salem.
\qed

\begin{Remark}
One can exhibit explicit automorphisms on $X(p)$ for odd $p\equiv 2\mod 3$
already as a ninefold composition of our automorphisms $\tau, \tau^\omega,\tau', \iota$,
but their characteristic polynomials turn out more complicated than $\mu(g^*)$.
On the other hand,
we can exhibit 
a slightly less complicated automorphism of Salem degree 20,
\[
g' = \tau'\circ\imath\circ\tau'\circ\imath
\circ\tau'\circ\imath\circ\tau'\circ\imath\circ\tau\in\Aut(X(p)),
\]
which does not lift to any characteristic zero model of $X(p)$.
This has characteristic polynomial $\mu(g'^*)=(x-1)(x+1)\psi(x)$
for some Salem polynomial $\psi(x)$ of degree $20$.
Since the latter factor has to be attained on $\NS$ of any characteristic zero lift,
one can apply Lemma \ref{lem:T} to the remaining linear factors of $\mu(g'^*)$ 
to establish a contradiction against lifting.
We emphasise, though, that this non-lifting argument does not apply to $g'^2$.
%
%
%
\end{Remark}

\begin{Remark}
The polynomial $\mu(g^*)$ is also irreducible over $\Q$
for $p=2$, i.e. $n=0$.
However, it is not immediate to deduce that it is the characteristic polynomial
of an automorphism on $X(2)$,
since the fibration $\pi'$ degenerates, see Remark \ref{rem:p=2}.
\end{Remark}

\subsection*{Acknowledgements}

The isotrivial fibrations first came up in discussions with Tetsuji Shioda
to whom I am greatly indebted.
I benefitted greatly from comments by Simon Brandhorst,
 H\'el\`ene Esnault and Keiji Oguiso
on an earlier version of the paper.
Thanks also to V\'ictor Gonz\'alez-Alonso and Jaap Top for discussions on the subject.

\end{document}